\documentclass[12pt,fleqn]{article}
\usepackage{authblk}
\usepackage{orcidlink}
\usepackage[misc]{ifsym}

\usepackage{mathtools} 

\usepackage[utf8]{inputenc}
\usepackage{paralist}
\usepackage[english]{babel}

\usepackage{amsthm}
\usepackage{amssymb}
\usepackage{amsmath}
\usepackage{paralist}
\usepackage{hyperref}

\def\IN{{\mathbb N}}

\def\dom{\mathop{\rm dom}}

\def\SigmaS{\Sigma^{\ast}} 
\def\SigmaN{\Sigma^{\IN}} 

\newtheorem{theorem}{Theorem}[section]
\newtheorem{lemma}[theorem]{Lemma}
\newtheorem{proposition}[theorem]{Proposition}
\newtheorem{corollary}[theorem]{Corollary}
\newtheorem*{theorem-main}{Theorem \ref{theorem:main-eq}}

\theoremstyle{definition}
\newtheorem{definition}[theorem]{Definition}
\newtheorem{example}[theorem]{Example}
\newtheorem{examples}[theorem]{Examples}

\theoremstyle{remark}
\newtheorem{remark}[theorem]{Remark}

\begin{document}
	
	\title{Strong Kurtz Randomness and Binary Expansions of Reordered Computable Numbers}
	\author{Peter Hertling
		\href{mailto:peter.hertling@unibw.de}{\Letter}
		\orcidlink{0000-0002-4442-6711}
	}
	\author{Philip Janicki
		\href{mailto:philip.janicki@unibw.de}{\Letter}
		\orcidlink{0009-0008-9063-028X}
	}
	\affil{Fakultät für Informatik \\
		Universität der Bundeswehr München \\
	}
	\date{\today}
	\maketitle

\begin{abstract}
A real number is called \emph{left-computable} if there exists a computable increasing sequence of rational numbers converging to it. In this article we investigate the Kolmogorov complexity and the binary expansions of a very specific subset of the left-computable numbers. We show in our main result that a real number is reordered computable if, and only if, it is left-computable and not strongly Kurtz random. In preparation of this, we characterize strong Kurtz randomness by a suitable notion of randomness tests.

We also look at the binary expansions of reordered computable numbers and clarify whether they can be immune, hyperimmune, hyperhyperimmune, strongly hyperhyperimmune, or cohesive.
Then, we investigate the effective Hausdorff and packing dimensions of reordered computable numbers. Finally, we have a short look at regular reals in the context of immunity properties, Kolmogorov complexity and (strong) Kurtz randomness.
\end{abstract}
		
	
\maketitle

\section{Introduction}\label{section:introduction}

In contrast to classical computability theory, computable analysis focuses on the computability of real-valued functions and the effectiveness of real numbers. There are three important subsets of the real numbers, which are also mentioned in this article.
A real number is called \emph{computable} if there exists a computable sequence of rational numbers converging computably to it. A real number is called \emph{left-computable} if there exists a computable increasing sequence of rational numbers converging to it. 
A real number is called \emph{computably approximable} if there exists a computable sequence of rational numbers converging to it. 
Every computable number is left-computable and every left-computable number is computably approximable, but none of these implications can be reversed.
The left-computable numbers play a very important role in algorithmic information theory. In this area, the real numbers are often represented by dyadic series. In this article we are interested in a specific subset of the left-computable numbers, namely the \emph{reordered computable} numbers, which have been introduced by the second author \cite{Jan2024}.

This article consists of seven sections. After the introduction in Section~\ref{section:introduction} and some preliminaries in Section \ref{section:preliminaries}, Section~\ref{section:strongKurtzrandom} is about strong Kurtz randomness. This randomness notion was introduced by Stephan and Wu~\cite{SW2005}, is defined in terms of prefix-free Kolmogorov complexity and is strictly between the well-known notions of Martin-Löf randomness and Kurtz randomness. We will give a characterization of strong Kurtz randomness by suitable randomness tests. After this, we will present our main result in Section~\ref{section:mainresult}. We will show that the reordered computable numbers are exactly the left-computable numbers which are not strongly Kurtz random. This result is interesting, since both concepts have been developed independently. The reordered computable numbers were defined in the context of computable analysis, while strong Kurtz randomness was defined in the context of algorithmic information theory.

Recently \cite{HJ2025}, the authors looked at the binary expansions of reordered computable numbers and clarified whether or not they can be immune or hyperimmune. In Section~\ref{section:binaryexpansions} we take this topic up again and extend our analysis to the notions hyperhyperimmune, strongly hyperhyperimmune, and cohesive.
It turns out that by using our main theorem and a number of well-known facts we can obtain a complete picture of the situation, obtaining shorter proofs of results in \cite{HJ2025} and new results. 

In the following short section~\ref{section:effectivedimension} we consider the question what are the possible effective dimensions that a reordered computable number can have. In Section~\ref{section:regular}, we have a short look at regular reals in the context of immunity properties, Kolmogorov complexity and (strong) Kurtz randomness. The regular reals, introduced by Wu \cite{Wu2005}, form a proper subset of the reordered computable numbers. We summarize the known results and additionally deduce from two old results that regular reals cannot be bi-immune. Furthermore, we show that regular reals are infinitely often $K$-trivial.

We thank the participants of the conference CiE 2025 in Lisbon whose questions after the presentation of 	\cite{HJ2025} by the second author motivated our further work on reordered computable numbers and in particular on the questions treated in Sections \ref{section:binaryexpansions} and \ref{section:effectivedimension}.

\section{Preliminaries}\label{section:preliminaries}

Let $\IN=\{0,1,2,\ldots\}$ be the set of natural numbers. 
For $i,j\in\IN$ we define $\langle i,j\rangle:= \frac{(i+j)(i+j+1)}{2}+j$.
Then $\langle\cdot,\cdot\rangle:\IN^2\to\IN$ is a computable bijection.
Let $\Sigma:=\{0,1\}$ be the binary alphabet. 
By $\SigmaS$ we denote the set of all binary strings.
By $\SigmaN$ we denote the set of all binary sequences.
Every binary sequence $x\in\SigmaN$ is a binary expansion of the real number $0.x=\sum_{i\in\IN} x(i)\cdot 2^{-(i+1)}$ in the interval $[0,1]$.
Quite often we will silently identify a binary sequence with the real number it represents and simply call binary sequences \emph{reals}.
For a binary string $\tau\in\SigmaS$ we sometimes write $0.\tau$ instead of $0.\tau0^\omega$.
And we let $I_\tau:=\{x\in [0,1] \mid 0.\tau0^\omega \leq x \leq 0.\tau1^\omega\}$ be the interval consisting of all real numbers in the unit interval
that have a binary representation starting with $\tau$.

For a set $A\subseteq\IN$ we write $\overline{A}$ for its complement $\IN\setminus A$.
For a set $A\subseteq\IN$ and a number $n\in\IN$ we write as usual $A\restriction n := A \cap\{0,\ldots,n-1\}$.
Sometimes, in order to simplify the notation, we will identify subsets of the natural numbers with binary sequences.
For example, we call $A\restriction n$ the \emph{prefix of length $n$ of $A$}.
It is well-known that a set $A\subseteq\IN$ is computable if and only if the real number $x_A:= \sum_{j\in A} 2^{-(j+1)}$ is computable. If $A$ is only assumed to be computably enumerable, then $x_A$ is a left-computable number; the converse is not true as pointed out by Jockusch (see Soare~\cite{Soa1969}). We say that a real number $x \in \left[0,1\right]$ is \emph{strongly left-computable} if there exists a computably enumerable set $A \subseteq \IN$ with $x_A = x$.

For any function $f \colon {\subseteq\IN}\to\IN$ that has either a finite domain or tends to infinity, we can define the function $u_f:\IN\to\IN$ by 
\[u_f(m):=\left|\{k\in\dom(f)\mid f(k)=m\}\right|.\]
Note that if $f \colon \IN\to\IN$ is a function such that the series $\sum_k 2^{-f(k)}$ converges, then $f$ tends to infinity.
For any real number $x > 0$ we will call a function $f \colon \IN \to \IN$ with $\sum_{k=0}^{\infty} 2^{-f(k)} = x$ a \emph{name} of $x$. Note that $x$ is left-computable if, and only if, $x$ has a computable name.

A sequence $(a_n)_n$ of real numbers is said to \emph{converge computably} if it converges and if there exists a computable function $g:\IN\to\IN$ such that, for all $m,n$, the assumption $m\geq g(n)$ implies $\left|a_m - \lim_{k\to\infty}a_k\right|\leq 2^{-n}$.

A computable function $M \colon {\subseteq\SigmaS} \to \SigmaS$ with prefix-free domain will be called a \emph{prefix-free machine}.
If $M$ is a prefix-free machine, then for any $\tau\in\SigmaS$ we define $K_M(\tau):=\min\{|\sigma| \mid \sigma\in\dom(M) \land M(\sigma)=\tau\}$, where $\min(\emptyset)=\infty$.
A prefix-free machine $U$ is called \emph{optimal} if for every prefix-free machine there exists a constant $c\in\IN$ with $K_U(\tau)\leq K_M(\tau)+c$ for all $\tau\in\SigmaS$. It is well-known that there exists an optimal prefix-free machine.
For the remainder of this article we fix some optimal prefix-free machine $U$ and write $K(\tau)$ for $K_U(\tau)$.
For any binary string $\tau$ we call $K(\tau)$ its \emph{Kolmogorov complexity}.

Let $(W_e)_e$ be a standard enumeration of all computably enumerable subsets of $\IN$, and let $\nu_{\SigmaS} \colon \IN\to\SigmaS$ be the bijection that enumerates the binary strings in length-lexicographical order.
Then $(V_e)_e$ defined by $V_e:=\{\nu_{\SigmaS}(n)\mid n\in W_e\}$ is a standard enumeration of all computably enumerable sets of binary strings.

\section{Strong Kurtz Randomness}
\label{section:strongKurtzrandom}

One of the most important topics in algorithmic information theory is the study of randomness of binary sequences. Perhaps the most important and central randomness notion in this field is Martin-Löf randomness. We review the definition first.

\begin{definition}
\begin{enumerate}
\item
We call a function $f \colon \IN \to \IN$ a \emph{Martin-L{\"o}f test} if $f$ is computable and
$\sum_{\tau\in V_{f(n)}} 2^{-|\tau|} \leq 2^{-n}$, for all $n$.
\item
We say that a Martin-L{\"o}f test $f:\IN\to\IN$ \emph{covers} a binary sequence $x \in \SigmaN$ if for every $n$ the set $V_{f(n)}$ contains a prefix of $x$.
\item
A binary sequence $x \in \SigmaN$ is called \emph{Martin-L{\"o}f random} if there is no Martin-L{\"o}f test that covers $x$.
\end{enumerate}
\end{definition}

It is well-known that a binary sequence $x \in \SigmaN$ is Martin-L{\"o}f random if, and only if, there exists a constant $c \in \IN$ with $K(x \restriction n) > n - c$ for every $n \in \IN$. There are many other randomness notions which are strictly weaker than Martin-Löf randomness. One prominent example is Kurtz randomness, which will be reviewed at the end of this section as well. Stephan and Wu~\cite[Definition 2]{SW2005} introduced the following randomness notion as a strengthening of Kurtz randomness.

\begin{definition}[Stephan, Wu \cite{SW2005}]
A binary sequence $x \in \SigmaN$ is called \emph{strongly Kurtz random} if there is no computable function $r\colon \IN\to\IN$ with
\begin{equation*}
K(x \restriction r(n)) \leq r(n) - n
\end{equation*}
for every $n \in \IN$.
\end{definition}

\begin{remark}\label{remark:lnc}
	Every Martin-L{\"o}f random binary sequence is strongly Kurtz random, but the converse is not true. The real numbers that are both left-computable and nearly computable but not computable serve as counterexamples. As in~\cite{HJ2024}, we call a sequence of real numbers $(x_n)_n$ \emph{nearly computably convergent} if it converges and, for every computable increasing function $s \colon \IN \to \IN$, the sequence $\left(x_{s(n+1)} - x_{s(n)}\right)_n$ converges computably to $0$, and we call a real number \emph{nearly computable} if there exists a computable sequence of rational numbers which converges nearly computably to it. Every computable number is nearly computable, but the converse is not true. It follows from a theorem by Downey and LaForte \cite[Theorem 3]{DL2002} that there exists a left-computable number which is nearly computable but not computable. Stephan and Wu \cite{SW2005} showed that these numbers are strongly Kurtz random but not Martin-L{\"o}f random.
\end{remark}

Strong Kurtz randomness can also be described by a suitable notion of a randomness test.

\begin{definition}
We call a function $f \colon \IN \to \IN$ a \emph{strong Kurtz test} if it is a Martin-L{\"o}f test and for every $n$ all strings in $V_{f(n)}$ have the same length.
\end{definition}

Of course, if $f$ is a strong Kurtz test then the set $V_{f(n)}$ is finite for every $n\in\IN$.

\begin{theorem}
\label{satz:non-strongly-kurtz-random-strong-kurtz-test}
A binary sequence is strongly Kurtz random if, and only if, it is not covered by any strong Kurtz test.
\end{theorem}

\begin{proof}
Let $x \in \SigmaN$.
       
$\implies$: Suppose that $f$ is a strong Kurtz test covering $x$.
Define the computable function $g \colon \IN \to \IN$ by $g(n) := f(2n+1)$ for all $n \in \IN$. Then, we have
\begin{align*}
      \sum_{n \in \IN}^{} \sum_{\tau \in V_{g(n)}}^{} 2^{-(\left|\tau\right| - n)}
      =& \sum_{n \in \IN}^{} 2^n \cdot \sum_{\tau \in V_{f(2n+1)}}^{} 2^{-\left|\tau\right|} \\
       &\leq \sum_{n \in \IN}^{} 2^n \cdot 2^{-(2n+1)} \\
       &= \sum_{n \in \IN}^{} 2^{-(n+1)} \\
       &= 1.
\end{align*}
By the effective Kraft-Chaitin theorem~\cite[Theorem 2.2.17]{Nie2009}, there exists a prefix-free machine $M$ such that, for every $n \in \IN$ and every $\tau \in V_{g(n)}$, there is some $\sigma \in \dom(M)$ with $M(\sigma) = \tau$ and $\left|\sigma\right| = \left|\tau\right| - n$.
Let $c \in \IN$ be a constant with $K(\tau) \leq K_M(\tau) + c$ for all $\tau\in \SigmaS$.
Let $r\colon\IN\to\IN$ be the function defined by $r(n):=|\tau|$ for any $\tau\in V_{g(n+c)}$, for all $n$. This function is computable.
Let us consider some $n\in\IN$ and set $\tau:=x \restriction r(n)$.
Then $\tau\in V_{g(n+c)}$. Let $\sigma \in \dom(M)$ be some string with $M(\sigma) = \tau$ and $\left|\sigma\right| = \left|\tau\right| - (n+c)$. We obtain
\begin{align*}
			K(x \restriction r(n))
      &\leq K_M(x \restriction r(n)) + c \\
      &= K_M(\tau) + c \\
      &\leq \left|\sigma\right| + c \\
      &= \left|\tau\right| - (n+c) + c \\
      &= r(n) - n,
\end{align*}
showing that $x$ is not strongly Kurtz random.
              
$\impliedby$:
Now suppose that $x$ is not strongly Kurtz random. Fix some computable function $r \colon \IN \to \IN$ such that $K(x \restriction r(n)) \leq r(n) - n$ for every $n \in \IN$.
Note that this implies $r(n)>n$ for all $n$.
Let $f:\IN\to\IN$ be a computable function such that
\begin{equation*}
     V_{f(n)} := \{ U(\sigma) \mid \sigma \in\dom(U) \land \left|U(\sigma)\right| = r(n) \land \left|\sigma\right| \leq r(n) - n \},
\end{equation*}
for all $n\in\IN$. We claim that $f$ is a strong Kurtz test covering $x$.
It is clear that $|\tau|=r(n)$ for all $\tau\in V_{f(n)}$ and $n\in\IN$.
Furthermore, for any $n$ the set $\dom(U)$ contains less than $2^{r(n)-n}$ strings of length at most $r(n)-n$. Hence, the set $V_{f(n)}$ contains less than $2^{r(n)-n}$ strings.
We obtain
\begin{equation*}
      \sum_{\tau \in V_{f(n)}} 2^{-\left|\tau\right|} < 2^{r(n)-n} \cdot 2^{-r(n)} = 2^{-n}
\end{equation*}
for every $n \in \IN$.
Finally, let us consider some $n \in \IN$. For $\tau:=x \restriction r(n)$ we have $|\tau|=r(n)$ and by assumption $K(\tau) \leq r(n)-n$. This implies $\tau\in V_{f(n)}$.
Therefore, $f$ is a strong Kurtz test covering $x$.
\end{proof}

\begin{definition}
A binary sequence $x\in\SigmaN$ is called \emph{Kurtz random}\footnote{These sequences are also called \emph{weakly 1-random} \cite{DH2010} or \emph{weakly random} \cite{Nie2009}.} if every computably enumerable, prefix-free set $S\subseteq\SigmaS$ with $\sum_{\sigma\in S} 2^{-|\sigma|}=1$ contains a prefix of $x$.
\end{definition}

By modifying slightly the Kurtz null tests by Wang~\cite{Wan1996} one obtains the following characterization.

\begin{lemma}
A binary sequence $x\in\SigmaN$ is not Kurtz random if, and only if, there is a strong Kurtz test $f$ covering $x$ such that
the function $n\mapsto |V_{f(n)}|$ is computable. 
\end{lemma}
	
Thus, any strongly Kurtz random binary sequence is indeed Kurtz random. The converse is not true.

\begin{proposition}
There exists a Kurtz random binary sequence that is not strongly Kurtz random.
\end{proposition}

\begin{proof}
It is known \cite[Remark 7.4.17]{Nie2009} that there exists a left-computable real $x$ which is partial computably random with $K(x \restriction n) \leq \log(n)^2 + c$ for some $c \in \IN$ and for all $n \in \IN$. Hence, $x$ is not strongly Kurtz random. But partial computably random sequences are Kurtz random~\cite{Nie2009}.
\end{proof}

\section{Main Result}
\label{section:mainresult}

In this section we show that there is a surprising connection between strong Kurtz randomness and the class of reordered computable numbers, introduced recently by the second author~\cite[Definition 7]{Jan2024}.

\begin{definition}[Janicki \cite{Jan2024}]
	A real number $x$ is called \emph{reordered computable} if $x$ has a computable name $f \colon \IN \to \IN$ such that the series $\sum_{k}^{} u_f(k) \cdot 2^{-f(k)}$ converges computably.
\end{definition}

It has recently been shown by Janicki \cite[Corollary 24]{Jan2024} that a reordered computable number cannot be Martin-Löf random. 
The following result strengthens this and gives an important characterization of reordered computable numbers. The proof of the $\impliedby$-implication is very similar to the proof of \cite[Theorem 2]{HJ2025}. 

\begin{theorem}\label{theorem:main-eq}
For a real number $x\in \left]0,1\right]$ the following are equivalent:
\begin{itemize}
\item
$x$ is reordered computable.
\item
$x$ is left-computable and not strongly Kurtz random.
\end{itemize}
\end{theorem}

\begin{proof}
Let $x \in \left]0,1\right]$.

$\implies$:
Suppose that $x$ is reordered computable. It is clear that $x$ is left-computable. We can assume that $x$ is not a dyadic number, otherwise $x$ is certainly not strongly Kurtz random.
Fix some computable name $f \colon \IN \to \IN$ for $x$ witnessing the reordered computability of $x$.
Furthermore, fix some computable function $r \colon \IN \to \IN$ with
$\sum_{k=r(n)}^{\infty} u_f(k) \cdot 2^{-k} \leq 2^{-n}$ for all $n \in \IN$.
We can assume without loss of generality that $r(0)>f(0)$ and that $r$ is increasing.
We define an increasing computable function $s\colon\IN\to\IN$ by $s(n):=r(n+3)+n+2$.
Let $(x_n)_n$ be the sequence defined by $x_n := \sum_{k=0}^{n-1} 2^{-f(k)}$, which is a computable increasing sequence of rational numbers converging to $x$. Note that $x_0=0$.

Before we continue with the formal proof, we try to explain the idea behind the following construction.
We shall define a strong Kurtz test that covers $x$.
The $n$-th level of the Kurtz test will consist of strings of length $s(n)$.
Therefore, in a sequence of stages $t$ for each $n$ we are going to define a finite collection of intervals
of the form $\left]x_t,x_t+2^{-s(n)}\right[$.
In stage $0$ we start by adding the interval $\left]0,2^{-s(n)}\right[$ to the $n$-th collection, for every $n$.
In any stage $t>0$ we add the interval $\left]x_t,x_t+2^{-s(n)}\right[$ to the $n$-th collection
if the sequence $(x_i)_i$ has moved by more than $2^{-s(n)}$ to the right since the previous stage
in which we added an interval to the $n$-th collection.
It is clear that in this way for any $n$ we construct a collection of pairwise disjoint intervals of the form
$\left]x_t,x_t+2^{-s(n)}\right[$, for certain $t$. This collection is finite because all intervals in it are
subsets of the interval $\left]0,x+2^{-s(n)}\right[$.
The $n$-th level of the strong Kurtz test will be essentially the union of these 
finitely many intervals of length $2^{-s(n)}$
(this is not true exactly, but right now, for the sketch of the idea, let us assume this).
Why is this a strong Kurtz test? 
For each $n$ there should be an interval of the form $\left]x_t,x_t+2^{-s(n)}\right[$
that contains the limit $x$.
Indeed, this will be true for the largest $t$ such that the interval $\left]x_t,x_t+2^{-s(n)}\right[$
is an element of the $n$-th collection
because for the largest such $t$ there will be no $t'>t$ with $x_{t'} - x_t>2^{-s(n)}$
(otherwise for such a $t'$ we would add the interval $\left]x_{t'},x_{t'}+2^{-s(n)}\right[$ to the $n$-th collection).
Furthermore, the total measure of the $n$-th level of the Kurtz test should be at most $2^{-n}$.
As for any $n$ the intervals in the $n$-th collection are pairwise disjoint
it is sufficient to show that there will not be too many stages
$t$ in which the interval $\left]x_t,x_t+2^{-s(n)}\right[$
is added to the $n$-th collection.
Note that this happens only when $x_t-x_{t'}>2^{-s(n)}$ where $t'$ is the previous stage
in which the corresponding interval $\left]x_{t'},x_{t'}+2^{-s(n)}\right[$ was added to the $n$-th collection.
So, we have to count the number of such stages $t$.
For counting the number of such stages $t$ we look at the size of the additive terms
$2^{-f(i)}$ for $t'\leq i < t$.
On the one hand, if there is such an $i$ with $f(i)<s(n)$ 
then the sequence has moved to the right by at least the large value $2^{-f(i)}$.
But, there will not be too many additive terms $2^{-f(i)}$ for small $f(i)$
because for small $f(i)$ the value $2^{-f(i)}$ is large and the sum of these additive terms is bounded by $x$.
Thus, there cannot be too many stages $t$ such that there is an $i$ with $t'\leq i < t$ and $f(i)<s(n)$.
On the other hand, if all the values $f(i)$ for $t' \leq i < t$ are large (then the additive terms $2^{-f(i)}$
are small) then the estimate
$\sum_{k=r(n)}^{\infty} u_f(k) \cdot 2^{-k} \leq 2^{-n}$ will be crucial. It says that the total movement
implied by large values $f(i)$ (thus, by small values $2^{-f(i)}$) is not that large.
Thus, there cannot be too many stages $t$ such that $x_t - x_{t'}> 2^{-s(n)}$
and all $i$ with $t'\leq i < t$ satisfy $s(n) \leq f(i)$.
Actually, in the formal counting argument below we shall need to distinguish three and not two cases.
But, let us come back to the formal construction. 

Let $\mathcal{I}$ be the set whose elements are the empty interval and all nonempty open intervals with dyadic endpoints.
We recursively compute a function $p \colon \IN^2 \to \IN$ 
(where we write $p(n)[t]$ for $p(n,t)$)
and a double sequence $J\colon\IN^2\to\mathcal{I}$ of intervals in $\mathcal{I}$
(where we write $J(n)[t]$ for $J(n,t)$).
In stage $t$ we shall compute $p(n)[t]$ and $J(n)[t]$ for all $n\in\IN$.
	
In stage $t=0$, we set $p(n)[0] := 0$ and $J(n)[0] := \left]0,2^{-s(n)}\right[$
for all $n \in \IN$. 

In any stage $t>0$, for $n\in\IN$ we set
\begin{eqnarray*}
  \begin{pmatrix} J(n)[t] \\ p(n)[t] \end{pmatrix}
	  := \begin{cases}
		\begin{pmatrix} \left]x_t, x_t  + 2^{-s(n)}\right[ \\ t \end{pmatrix} 
		  & \text{if } n \leq t \text{ and } x_t > x_{p(n)[t-1]} + 2^{-s(n)}, \\
 		\begin{pmatrix} \emptyset \\ p(n)[t-1] \end{pmatrix} 
	    & \text{otherwise},
		\end{cases}
\end{eqnarray*}
This ends the description of stage $t$ and of the definition of $p$ and $J$.
Finally, let $g\colon\IN\to\SigmaS$ be a computable function such that
\[ V_{g(n)} = \left\{\tau \in \Sigma^{s(n)} 
  ~\bigg{|}~ I_\tau \cap \bigcup_{t=0}^\infty J(n)[t] \neq \emptyset \right\}, \]
for all $n\in\IN$.
This ends the construction.

The following assertions are clear:
\begin{itemize}
\item
For each $n\in\IN$ and $t\in\IN$ the interval $J(n)[t]$ is either empty or has length $2^{-s(n)}$.
\item
For each $n\in\IN$ and $t\in\IN$ the number $p(n)[t]$ is the largest stage $t'\leq t$
such that $J(n)[t']$ is not empty.
\end{itemize}
We claim that $g$ is a strong Kurtz test covering (a binary representation of) $x$.
Let us fix a number $n$. We have to show three assertions:
\begin{enumerate}
\item
All strings in the set $V_{g(n)}$ have the same length.
\item
$\sum_{\tau\in V_{g(n)}} 2^{-|\tau|} \leq 2^{-n}$.
\item
There is a string $\tau\in V_{g(n)}$ with $x\in I_\tau$.
\end{enumerate}

The first assertion is clear.

We come to the proof of the second assertion.
Let us consider a number $n\in\IN$. 
Remember that for any $t\in\IN$ the interval $J(n)[t]$
is either empty or has length $2^{-s(n)}$.
For any open interval $J(n)[t]$ of length $2^{-s(n)}$ there are most two strings $\tau$ of length $s(n)$
with $I_\tau \cap J(n)[t] \neq\emptyset$.
Thus, it is sufficient to show that there at most $2^{s(n)-n-1} = 2^{r(n+3)+1}$ stages $t$ such that $J(n)[t]$ is not empty.
By construction, for any $t\in\IN$, the interval $J(n)[t]$ can be nonempty only if $t=0$ or if $t>0$ and 
\[ 2^{-s(n)} < x_t - x_{p(n)[t-1]} = \sum_{k \in \{p(n)[t-1], \dots, t-1\}} 2^{-f(k)} . \]
And for $t>0$ this can happen only if there is some $k\in\{p(n)[t-1], \dots, t-1\}$ with $f(k)<s(n)$ or if $\sum_{\substack{k \in \{p(n)[t-1], \ldots, t-1\}, \\ s(n) \leq f(k)}} 2^{-f(k)} > 2^{-s(n)}$ holds. The first case happens at most $\sum_{k=0}^{s(n)-1} u_f(k)$ times, the second case at most $\left\lfloor \frac{\sum_{k=s(n)}^\infty u_f(k) \cdot 2^{-k}}{2^{-s(n)}} \right\rfloor$ times. 
So, the number of stages $t$ such that $J(n)[t]$ is nonempty is bounded by
	\begin{align*}
		&1 + \sum_{k=0}^{s(n)-1} u_f(k) + \left\lfloor \frac{\sum_{k=s(n)}^\infty u_f(k) \cdot 2^{-k}}{2^{-s(n)}} \right\rfloor \\
		&\leq 1+ \sum_{k=0}^{r(n+3)-1} u_f(k) + \sum_{k=r(n+3)}^{s(n)-1} u_f(k) +  2^{s(n)}\cdot \sum_{k=s(n)}^\infty u_f(k) \cdot 2^{-k}.
\end{align*}
(the additive term $1$ at the beginning of this sum is due to the stage $t=0$).
We give separate estimates for the three main terms in this sum:
\[ 		\sum_{k=0}^{r(n+3)-1} u_f(k)
    \leq 2^{r(n+3)-1} \cdot \sum_{k=0}^{r(n+3)-1} u_f(k) \cdot 2^{-k}
		\leq 2^{r(n+3)-1} \cdot x  \]
and
\begin{eqnarray*}
    \sum_{k=r(n+3)}^{s(n)-1} u_f(k)
    &\leq & 2^{r(n+3)+n+1} \cdot \sum_{k=r(n+3)}^{r(n+3)+n+1} u_f(k) \cdot 2^{-k} \\
		&\leq & 2^{r(n+3)+n+1} \cdot 2^{-(n+3)}
		= 2^{r(n+3)-2} 
\end{eqnarray*}
and
\begin{eqnarray*}   
   2^{s(n)}\cdot \sum_{k=s(n)}^\infty u_f(k) \cdot 2^{-k}
   &\leq & 2^{s(n)} \cdot \sum_{k=r(n+3)}^\infty u_f(k) \cdot 2^{-k} \\
	 &\leq & 2^{r(n+3)+n+2} \cdot 2^{-(n+3)} 
	 = 2^{r(n+3)-1} .
\end{eqnarray*}
So, for the number of stages $t$ such that $J(n)[t]$ is nonempty we obtain the upper bound
\[ 1 + 2^{r(n+3)-1} \cdot x + 2^{r(n+3)-2} + 2^{r(n+3)-1} < 2^{r(n+3)+1} \]
(remember that $x\leq 1$). 

We come to the proof of the third assertion.
Let us consider a number $n\in\IN$. 
We have to show that there is a string $\tau\in V_{g(n)}$ with $x\in I_\tau$.
It is sufficient to show that there is some stage $t$ such that $x\in J(n)[t]$.
In the proof of the second assertion we have seen that $J(n)[t]$ is nonempty only for finitely many $t$.
That implies that the nondecreasing sequence $(p(n)[t])_t$ is bounded and, therefore,
ultimately constant. 
Let $t_\infty$ be the smallest stage with $p(n)[t_\infty] = \lim_{t \to \infty} p(n)[t]$.
Then $t_\infty = p(n)[t_\infty] = \lim_{t \to \infty} p(n)[t]$
and $J(n)[t_\infty] = \left]x_{t_\infty},x_{t_\infty}+2^{-s(n)}\right[$.
For any $t'\geq \max\{n, t_\infty+1\}$ we have
$x_{t'} \leq x_{p(n)[t_\infty]} + 2^{-s(n)} = x_{t_\infty} + 2^{-s(n)}$.
But this implies $x_{t_\infty} < x \leq x_{t_\infty} + 2^{-s(n)}$.
Remember that we assumed that $x$ is not a dyadic number,
hence we even have $x_{t_\infty} < x < x_{t_\infty} + 2^{-s(n)}$,
hence $x\in J(n)[t_\infty]$.

We have shown that $g$ is a strong Kurtz test covering $x$. Therefore, $x$ is not strongly Kurtz random.

$\impliedby$: Now suppose that $x$ is left-computable and not strongly Kurtz-random. 
We can assume that $x$ is not a dyadic number, otherwise $x$ is certainly reordered computable. Let $(x_n)_n$ be a computable increasing sequence of dyadic numbers converging to $x$ with $x_0 := 0$. Let $r \colon \IN \to \IN$ be a computable increasing function with $K(x \restriction r(n)) \leq r(n) - n$ for every $n \in \IN$. 
	
We claim that there is an increasing computable function $s \colon \IN \to \IN$ with the following property:
for all $t\in\IN$
\[ (\forall n \in \{0, \dots, t-1\})\ K(x_{s(t)} \restriction r(n)) \leq r(n) - n . \]
Indeed, we can define $s(0):=0$. Once suitable values $s(0),\ldots,s(t)$ have been found,
we search for a number $m>s(t)$ such that 
$(\forall n \in \{0, \dots, t\})\ K(x_{m} \restriction r(n)) \leq r(n) - n$. We will find such a number $m$ because
$(x_n)_n$ converges to $x$ and $K(x \restriction r(n)) \leq r(n) - n$ is true for every $n \in \IN$.
Then we set $s(t+1):=m$.
	
We want to show that $x$ is reordered computable. We define a computable function $f:\IN\to\IN$ witnessing this property in stages. In stage $0$ we start with the function with empty domain, and in any stage $t+1$ we append finitely many values as follows: Let $f[t]$ be the finite prefix of $f$ that is defined at the end of stage $t$. For each $t\in\IN$, let $D[t]\subseteq\IN$ be the uniquely determined finite set of natural numbers with $x_{D[t]} = x_{s(t+1)}-x_{s(t)}$.
	We fix the next $|D[t]|>0$ values of $f$ by demanding that these values should be exactly the values in the set $D[t]$ in increasing order. In this way a computable function $f:\IN\to\IN$ is defined.
	Obviously, for all $t\in\IN$ it satisfies 
	\[ x_{s(t)} = \sum_{m=0}^\infty u_{f[t]}(m) \cdot 2^{-m} \]
	(note that $u_{f[t]}(m)=0$ for almost all $m$)
	and therefore $x = \sum_{m=0}^\infty u_f(m) \cdot 2^{-m} = \sum_{k=0}^\infty 2^{-f(k)} $.
	Thus, $f$ is a computable name of $x$.
	We claim that the series $\sum_m u_f(m) \cdot 2^{-m}$ converges computably. In fact, we claim that
	\begin{equation*}
		\sum_{m=r(n)+1}^{\infty} u_f(m) \cdot 2^{-m} \leq (n+1)\cdot 2^{-n}
	\end{equation*}
	for all $n \in \IN$.
	
	Fix some arbitrary $n \in \IN$. 
  Let us first give a rough summary of the idea.
	We wish to give an upper bound for the sum $\sum_{m=r(n)+1}^{\infty} u_f(m) \cdot 2^{-m}$.
	Therefore, we look at the sequence $\left(\sum_{m=r(n)+1}^{\infty} u_{f[t]}(m) \cdot 2^{-m}\right)_t$.
	On the one hand, in each step from $t$ to $t+1$ this sum can increase at most by $2^{-r(n)}$.
	Thus, when we consider the largest multiple of $2^{-r(n)}$ below it, we see that
	we obtain a sequence of multiples of $2^{-r(n)}$ starting at zero and without gaps.
	So, the number of changes of this multiple is essentially equal to $\sum_{m=r(n)+1}^{\infty} u_f(m) \cdot 2^{-m}$.
	On the other hand, whenever in this way we reach the next multiple of $2^{-r(n)}$ 
	this causes a carry to position $r(n)-1$
	and, thus, causes a change in the prefix $x_{s(n)} \restriction r(n)$ of length $r(n)$ of $x_{s(n)}$.
	But, the upper bound for the Kolmogorov complexity $K(x_{s(n)} \restriction r(n))$
	implies that for $t>n$ there are considerably less than $2^{r(n)}$ possibilities for this prefix
	$x_{s(n)} \restriction r(n)$.
	Thus for increasing $t>n$ the sum $\sum_{m=r(n)+1}^{\infty} u_{f[t]}(m) \cdot 2^{-m}$ 
	cannot reach the next multiple of $2^{-r(n)}$ very often because that results in a change of
	the prefix of length $r(n)$.
	In this way we will obtain an upper bound on the limit $\sum_{m=r(n)+1}^{\infty} u_f(m) \cdot 2^{-m}$.
	
	Now we come to the detailed verification.
	For any $t>n$ we have $K(x_{s(t)} \restriction r(n)) \leq r(n) - n$.
  There are at most $2^{r(n)-n}-1$ strings $w$ of length $r(n)$ with 
	$K(w) \leq r(n)-n$. Hence, for $t>n$ there at most $2^{r(n)-n)}-1$ possibilities
	for the prefix $x_{s(t)}\restriction r(n)$ of length $r(n)$ of the binary representation of $x_{s(t)}$.
  Since the sequence $(x_{s(t)})_t$ is increasing,
	if $t_1<t_2$ satisfy $x_{s(t_2)}\restriction r(n) \neq x_{s(t_1)}\restriction r(n)$
	then for all $t_3\geq t_2$ we have $x_{s(t_3)}\restriction r(n) \neq x_{s(t_1)}\restriction r(n)$ as well.
	Hence, there are at most $2^{r(n)-n}-2$ stages $t > n$
	such that $x_{s(t+1)}\restriction r(n) \neq x_{s(t)}\restriction r(n)$.
	
	Such a change of the first $r(n)$ bits will certainly occur whenever a carry to the position $r(n)-1$ happens.
	By that we mean that the sum of the (small) additive terms of the form $2^{-f(i)}$ with $f(i)\geq r(n)+1$ reaches the next
	multiple of $2^{-r(n)}$ (note that the numerical contribution of a bit at position $r(n)-1$ in a binary representation
	of a number is $2^{-r(n)}$) and causes a change in the first $r(n)$ bits (which end with the bit in position $r(n)-1$). To make this precise, a carry to the position $r(n)-1$ happens in stage $t+1$ if
	$R[t+1]>R[t]$ where
the sequence $(R[t])_t$ is defined by
	\begin{align*}
		R[t] &:= \left\lfloor 2^{r(n)}\cdot \sum_{m=r(n)+1}^\infty u_{f[t]}(m)\cdot 2^{-m} \right\rfloor
	\end{align*}
	for all $t\in\IN$. We observe
	\begin{equation*}
		\sum_{m=r(n)+1}^\infty u_{f[t]}(m)\cdot 2^{-m} < (R[t] + 1) \cdot 2^{-r(n)}.
	\end{equation*}
	Note that the sum appearing in the definition of $R[t]$ is actually a finite sum because $u_{f[t]}(m)=0$ for almost all $m$. Obviously, $(R[t])_t$ is non-decreasing and bounded. Let $R[\infty] := \lim_{t\to\infty} R[t]$. Then we obtain the inequality
	\begin{equation*}
		\sum_{m=r(n)+1}^\infty u_f(m)\cdot 2^{-m} \leq (R[\infty] + 1) \cdot 2^{-r(n)}.
	\end{equation*}
	The next observation is crucial:
as for all $m,t\in\IN$ we have $u_{f[t+1]}(m) \leq u_{f[t]}(m)+1$, the difference between the sums $\sum_{m=r(n)+1}^\infty u_{f[t+1]}(m)\cdot 2^{-m}$ and $\sum_{m=r(n)+1}^\infty u_{f[t]}(m)\cdot 2^{-m}$ is smaller than $2^{-r(n)}$, which implies $R[t+1] \leq R[t] + 1$. Thus,
\[ R[t+1] > R[t] \iff R[t+1] = R[t]+1 . \]
Hence, the set \[ T:= \{t\in\IN\mid R[t+1] > R[t]\} \]
	is finite, and its cardinality is $|T| = R[\infty]$ (remember that we assume that $x$ is not a dyadic number).
	Note that the elements in $T$ are exactly the stages $t$ such that a carry to position $r(n)-1$ occurs in stage $t+1$.
As we explained above, this implies
$x_{s(t+1)}\restriction r(n) \neq x_{s(t)}\restriction r(n)$, that is, for $i$ increasing from $t$ to $t+1$,
the binary expansion of $x_{s(i)}$ changes on the first $r(n)$ bits.
Above we already noted that there are at most $2^{r(n)-n}-2$ stages $t > n$ where this can happen.
	By taking into account that $T$ may also contain stages $t \leq n$ 
	(we do not have any control over the prefix of length $r(n)$ of $x_{s(t)}$ for $t\leq n$; so, the stages $t\leq n$ have
	to be taken into account separately)
	we obtain 
	\begin{equation*}
		\left|T\right| \leq 2^{r(n)-n} - 2 + n + 1 =  2^{r(n)-n} + n - 1.
	\end{equation*} 
	Remember that $r$ is increasing, so we have $r(n)\geq n$ for all $n \in \IN$. Putting everything together, we obtain
	\begin{align*}
		\sum_{m=r(n)+1}^{\infty} u_f(m) \cdot 2^{-m}
		&\leq \left(R[\infty] + 1\right) \cdot 2^{-r(n)} \\
		&= \left(\left|T\right| + 1\right) \cdot 2^{-r(n)} \\
		&\leq (2^{r(n)-n} + n) \cdot 2^{-r(n)} \\
		&= 2^{-n} + n\cdot 2^{-r(n)} \\
		&\leq (n+1)\cdot 2^{-n},
	\end{align*}
	showing that $x$ is reordered computable.
\end{proof}

\section{Binary Expansions}
\label{section:binaryexpansions}

We are interested in the binary expansions of reordered computable numbers. For an infinite set $A\subseteq\IN$ let $p_A \colon \IN\to\IN$ be the uniquely determined increasing function with $p_A(\IN) = A$.

\begin{definition}\label{defi:immunity-definitions}
	Let $A \subseteq \IN$ be an infinite set.
	\begin{enumerate}[(1)]
		\item $A$ is called \emph{immune} if it does not have any infinite computably enumerable subset.
		\item $A$ is called \emph{hyperimmune} if there is no computable function $f \colon \IN \to \IN$ with $p_A(n) \leq f(n)$ for all $n \in \IN$.
		\item $A$ is called \emph{hyperhyperimmune} if there is no computable function $f \colon \IN \to \IN$ such that $(W_{f(n)})_n$ is a sequence of pairwise disjoint and finite sets with $A \cap W_{f(n)} \neq \emptyset$
		for all $n \in \IN$.
		\item $A$ is called \emph{strongly hyperhyperimmune} if there is no computable function $f \colon \IN \to \IN$ such that $(W_{f(n)})_n$ is a sequence of pairwise disjoint sets with $A \cap W_{f(n)} \neq \emptyset$
		for all $n \in \IN$.
		\item $A$ is called \emph{cohesive} if there is no number $e \in \IN$ such that both $A \cap W_e$ and $A \cap \overline{W_e}$ are infinite.
		\item $A$ is called \emph{bi-immune} if both $A$ and $\overline{A}$ are immune.
	\end{enumerate}
\end{definition}

It is well-known that for every $X \in \{1, \dots, 4\}$ the implication $(X+1) \implies (X)$ holds, but none of these implications can be reversed.
For more information, the reader is referred to \cite[Chapter XI.1]{Soa1987}. However, if we restrict ourselves to computably approximable sets, the notions of hyperhyperimmunity and strong hyperhyperimmunity become equivalent. A set $A\subseteq\IN$ is called \emph{computably approximable}\footnote{These sets are also called \emph{limit-computable} or \emph{$\Delta_2^0$}.} if there is a computable function $a \colon \IN^2 \to \{0,1\}$ with $A(i) = \lim_{t\to\infty} a(i, t)$ for every $i \in \IN$. It is well-known that a set $A \subseteq \IN$ is computably approximable if and only if the number $x_A$ is computably approximable. The following lemma has been stated by Cooper \cite[Page 600]{Coo1972}. For a proof of a stronger assertion the reader is referred to Rolletschek~\cite[Proposition 1]{Rol1983}.

\begin{lemma}[Cooper \cite{Coo1972}]
\label{lem:Delta^0_2+hii=>shhi}
	Let $A \subseteq \IN$ be a set which is both computably approximable and hyperhyperimmune. Then $A$ is even strongly hyperhyperimmune.
\end{lemma}

We call a real number $x \in \left[0,1\right]$ \emph{immune} if there exists an immune set $A \subseteq \IN$ with $x_A = x$. The other notions for sets as in Definition \ref{defi:immunity-definitions} are used accordingly on real numbers in $[0,1]$. The following corollary strengthens \cite[Theorem 2]{HJ2025}.

\begin{corollary}
\label{kor:bi-immune}
Every left-computable number in $[0,1]$ that is not bi-immune is reordered computable.
\end{corollary}

\begin{proof}
Let $x$ be a left-computable number that is not bi-immune.
As every Kurtz random set is bi-immune (this is due to Jockusch; see \cite[Theorem 7.2.23]{DH2010}), $x$ is not Kurtz random, hence, not strongly Kurtz random.
With Theorem~\ref{theorem:main-eq} we conclude that $x$ is reordered computable.
\end{proof}

After the presentation of the results in \cite{HJ2025} at the CiE 2025 conference in Lisbon, the second author was asked whether there are examples of reordered computable numbers which are hyperhyperimmune. We will answer this question in a very strong sense, showing that every left-computable number which is hyperhyperimmune is reordered computable and that there are such numbers.

\begin{lemma}\label{lem:shhi=>not-bi-immune}
	Let $A \subseteq \IN$ be a strongly hyperhyperimmune set. Then $\overline{A}$ is not immune.
\end{lemma}

\begin{proof}
	Fix some computable function $f\colon\IN\to\IN$ with 
	\begin{equation*}
		W_{f(i)} := \{ \left\langle i,j \right\rangle \mid j \in \IN \}
	\end{equation*}
	for every $i \in \IN$. All these computably enumerable sets are infinite and pairwise disjoint. Since $A$ is strongly hyperhyperimmune, there is some $k \in \IN$ with $A \cap W_{f(k)} = \emptyset$. Therefore, we have $W_{f(k)} \subseteq \overline{A}$, hence $\overline{A}$ is not immune.
\end{proof}

\begin{corollary}\label{corollary:hhi=>roc}
	Every left-computable number which is hyperhyperimmune is reordered computable.
\end{corollary}

\begin{proof}
	Let $x$ be a left-computable number which is hyperhyperimmune. Since $x$ is computably approximable, according to Lemma~\ref{lem:Delta^0_2+hii=>shhi} $x$ is even strongly hyperhyperimmune. Then, $x$ is not bi-immune, according to Lemma \ref{lem:shhi=>not-bi-immune}. With Corollary~\ref{kor:bi-immune} the assertion follows.
\end{proof}

Soare has shown~\cite[Theorem 3.1]{Soa1969} that there exists a cohesive set $A\subseteq\IN$ such that $x_A$ is a left-computable number. Using Corollary \ref{corollary:hhi=>roc}, we obtain the following corollary.

\begin{corollary}
	There exists a reordered computable number which is cohesive.
\end{corollary}

In the remainder of this section we will show that, restricted to the reordered computable numbers, the notions cohesive, hyperhyperimmune, hyperimmune and immune form a proper hierarchy. For arbitrary sets $A, B\subseteq\IN$ the set $A\oplus B\subseteq\IN$ is defined by
$
	A\oplus B := \{2n\mid n\in A\} \cup \{2n+1\mid n\in B\}.
$

\begin{proposition}[Rolletschek \cite{Rol1983}]\label{prop:closure-self-join}
	Let $A \subseteq \IN$ be an infinite set, and let $B := A \oplus A$.
	\begin{itemize}
		\item $A$ is immune if and only if $B$ is immune.
		\item $A$ is hyperimmune if and only if $B$ is hyperimmune.
		\item $A$ is strongly hyperhyperimmune if and only if $B$ is strongly hyperhyperimmune.
	\end{itemize}
\end{proposition}

\begin{lemma}\label{lem:self-join}
	Let $A \subseteq \IN$ be an infinite set, and let $B := A \oplus A$.
	\begin{enumerate}[(1)]
		\item $B$ is not cohesive.
		\item If $x_A$ is left-computable, then $x_B$ is reordered computable.
	\end{enumerate}
\end{lemma}

\begin{proof}
	\begin{enumerate}[(1)]
		\item Since $A$ is infinite, $B$ contains infinitely many even numbers and infinitely many odd numbers. Therefore, $B$ is not cohesive.
		\item It is clear that $x_B$ is left-computable as well. Since $x_B$ avoids the substring $010$ (and the substring $101$ as well), this number is not Kurtz random, hence, not strongly Kurtz random, hence, reordered computable, by Theorem~\ref{theorem:main-eq}.
		\qedhere
	\end{enumerate}
\end{proof}

\begin{theorem}
	There exists a reordered computable number which is strongly hyperhyperimmune but not cohesive.
\end{theorem}
	
\begin{proof}
As above, let $A \subseteq \IN$ be a cohesive set such that $x_A$ is left-computable \cite[Theorem 3.1]{Soa1969}.
In particular, $A$ is strongly hyperhyperimmune. Then, according to Proposition~\ref{prop:closure-self-join} and Lemma~\ref{lem:self-join}, the set $B := A \oplus A$ is strongly hyperhyperimmune but not cohesive, and $x_B$ is reordered computable.
\end{proof}

According to Lemma \ref{lem:Delta^0_2+hii=>shhi}, every computably approximable number which is hyperhyperimmune is even strongly hyperhyperimmune. Therefore, these two notions coincide among the reordered computable numbers. The next theorem strengthens \cite[Theorem 4]{HJ2025}. In its proof we shall use the notion of weak $1$-genericity.

\begin{definition} \;
	\begin{itemize}
		\item A subset $S\subseteq\SigmaN$ is called \emph{c.e.~open} if there is a c.e. set $L\subseteq\SigmaS$ such that $S=\{x\in\SigmaN\mid (\exists\sigma\in L) \, \sigma$ is a prefix of $x\}$.
		\item A binary sequence $x\in\SigmaN$ is called \emph{weakly $1$-generic} if every dense and c.e. open subset of $\SigmaN$ contains $x$.
	\end{itemize}
\end{definition}

\begin{remark}\label{remark:weakly-1-generic}
	It is easy to see that if $A \subseteq \IN$ is a set which is weakly $1$-generic, $\overline{A}$ is weakly $1$-generic as well. Furthermore, every weakly $1$-generic set is hyperimmune \cite[Lemma 1.8.48]{Nie2009}. Hoyrup showed \cite[Proposition 4.3.2]{Hoy2017} that a left-computable number which is nearly computable but not computable (see Remark \ref{remark:lnc}) is weakly $1$-generic.
\end{remark}

\begin{theorem}
	There exists a reordered computable number which is hyperimmune but not hyperhyperimmune.
\end{theorem}

\begin{proof}
	Let $A \subseteq \IN$ be such that the number $x_A$ is left-computable and nearly computable but not computable. Then, $x_A$ is weakly $1$-generic, hence both $A$ and $\overline{A}$  are hyperimmune (see Remark \ref{remark:weakly-1-generic}). In particular, $\overline{A}$ is immune. Then, according to Lemma \ref{lem:shhi=>not-bi-immune}, $A$ is not strongly hyperhyperimmune.

Now let $B := A \oplus A$. Then, $B$ is hyperimmune but not strongly hyperhyperimmune, due to Proposition \ref{prop:closure-self-join}. Since $B$ is computably approximable, this set is not hyperhyperimmune either. However, $x_B$ is reordered computable, according to Lemma \ref{lem:self-join}.
\end{proof}

The following theorem was stated already in~\cite[Theorem 5]{HJ2025}, and a direct construction was sketched there. But the following proof is much simpler. For the definition of the regular numbers the reader is referred to Section~\ref{section:regular}.

\begin{theorem}[{\cite[Theorem 5]{HJ2025}}]\label{theorem:roc+immune+non-hi}
	There exists a reordered computable number which is immune but not hyperimmune and not regular.
\end{theorem}

\begin{proof}
	Let $A \subseteq \IN$ be such that $x_A$ is a left-computable number which is Martin-Löf random. Then $A$ is immune but not hyperimmune. Let $B := A \oplus A$. Then, $x_B$ is a reordered computable number which is immune but not hyperimmune, according to Lemma \ref{lem:self-join} and Proposition \ref{prop:closure-self-join}. Furthermore, $x_B$ is not regular, since its Kolmogorov complexity grows linearly, while the Kolmogorov complexity of regular numbers is logarithmic (Proposition~\ref{proposition:regular-Kolmogorovcomplexity}).
\end{proof}

\section{Effective dimensions}
\label{section:effectivedimension}

After the presentation of the results in \cite{HJ2025} at the CiE 2025 conference in Lisbon, the second author was asked by Alexander Shen what are the possible effective Hausdorff dimensions that a reordered computable number can have.

\begin{definition}
Let $x \in \SigmaN$ be a binary sequence.
\begin{itemize}
\item The value $\dim(x) := \liminf_{n \to \infty} \frac{K(x \restriction n)}{n}$ is called the \emph{effective Hausdorff dimension} of $x$.
\item The value $\mathrm{Dim}(x) := \limsup_{n \to \infty} \frac{K(x \restriction n)}{n}$ is called the \emph{effective packing dimension} of $x$.
\end{itemize}
\end{definition}

It is clear that for any binary sequence $x$ we have
\begin{equation*}
	0 \leq \dim(x)\leq\mathrm{Dim}(x) \leq 1.
\end{equation*}

We start with some examples of well-known classes of left-computable numbers $x$ showing various different possible combinations of the values $\dim(x)$ and $\mathrm{Dim}(x)$.

\begin{examples}
\label{examples:effdim} \;
\begin{enumerate} 
		\item
		Chaitin's constant is defined by
		$\Omega := \sum_{\sigma \in \dom(U)} 2^{-\left|\sigma\right|}$
		where $U$ is some optimal prefix-free machine. This is a left-computable number which is Martin-Löf random.
		Hence, we have $\dim(\Omega) = \mathrm{Dim}(\Omega) = 1$.
		\item
		Let $U$ be an optimal prefix-free machine and $s \in \left]0,1\right[$ be a computable number. Then, $\Omega^s := \sum_{\sigma \in \dom(U)}^{} 2^{-\frac{\left|\sigma\right|}{s}}$ is a left-computable number with $\dim(\Omega^s) = \mathrm{Dim}(\Omega^s) = s$, as shown by Tadaki \cite[Theorem 3.2]{Tad2002}.
		\item
		A binary sequence $x \in \SigmaN$ is called \emph{$K$-trivial} if there exists a constant $c \in \IN$ such that $K(x \restriction n) \leq K(0^n) + c$ for all $n \in \IN$. As $K(0^n)$ grows only logarithmically with growing $n$, we have $\dim(x)=\mathrm{Dim}(x)=0$. Downey, Hirschfeldt, Nies and Stephan \cite[Theorem 3.1]{DHNS2003} proved that there exists a strongly left-computable number which is $K$-trivial and not computable.
		\item
		A binary sequence $x \in \SigmaN$ is called \emph{infinitely often $K$-trivial} if there exists a constant $c \in \IN$ such that $K(x \restriction n) \leq K(0^n) + c$ for infinitely many $n \in \IN$. Using the same argument as in the third example, one can show that $\dim(x) = 0$ for an infinitely often $K$-trivial real $x$.
		
		A left-computable number $x$ is called \emph{regainingly approximable} \cite{HHJ2024} if there exists a computable increasing sequence $(x_n)_n$ of rational numbers converging to $x$ with $x - x_n \leq 2^{-n}$ for infinitely many $n \in \IN$. Hertling, Hölzl and Janicki \cite[Proposition 6.1]{HHJ2024} showed that all regainingly approximable numbers in $[0,1]$ are infinitely often $K$-trivial. Furthermore, they showed~\cite[Theorem 6.2]{HHJ2024} that there exists a regainingly approximable number $x \in [0,1]$ with $K(x \restriction n) > n$ for infinitely many $n \in \IN$. Thus, such a number $x$ satisfies $\dim(x) = 0$ and $\mathrm{Dim}(x) = 1$.
\end{enumerate}
\end{examples}

These examples and the following observations are our answer to Shen's question.

\begin{theorem}
\label{theorem:Dim<1-notstronglyKurtzrandom}
Let $x \in \SigmaN$ be a binary sequence with $\mathrm{Dim}(x) < 1$. Then $x$ is not strongly Kurtz random.
\end{theorem}
 
\begin{proof}
Due to $\mathrm{Dim}(x) < 1$, there exists a natural number $m \in \IN$ with $K(x \restriction n) \leq \frac{m+1}{m+2} \cdot n$ for almost all $n \in \IN$.
Define the function $r \colon \IN\to \IN$ by $r(n) := (m+2) \cdot n$, which is computable. Then, we obtain
\begin{equation*}
K(x \restriction r(n)) \leq \frac{m+1}{m+2} \cdot r(n) = (m+1) \cdot n = r(n) - n
\end{equation*}
for almost every $n \in \IN$. Therefore, $x$ is not strongly Kurtz random.
\end{proof}

\begin{corollary}\label{kor:Dim<1=>roc}
Every left-computable real $x\in [0,1]$ with $\mathrm{Dim}(x) < 1$ is reordered computable.
\end{corollary}

\begin{theorem}
\label{theorem:dim1}
There exists a reordered computable real $x\in [0,1]$ with $\dim(x) = 1$.
\end{theorem}

\begin{proof}
Let $U$ be an optimal prefix-free machine. Then, the number $\Omega := \sum_{w \in \dom(U)}^{} 2^{-\left|w\right|}$ is left-computable and Martin-Löf random.
Hence, there exists a constant $c$ such that $K(\Omega \restriction n)> n-c$ for all $n$.
Now let $x$ be the real number whose binary expansion contains a zero at every position which is a square (where we start with position $1$ after the binary point).
The other positions are filled with the binary expansion of $\Omega$. So we have
\begin{equation*}
      x = 0. 0 \Omega(0) \Omega(1) 0 \Omega(2) \Omega(3) \Omega(4) \Omega(5) 0 \Omega(6) \dots,
\end{equation*}
which is a left-computable number. Since $x$ is not bi-immune, $x$ is reordered computable, according to Corollary~\ref{kor:bi-immune}. It is easy to verify that there exists a constant $c \in \IN$ with
\begin{equation*}
     K_U(x \restriction n) > n - \lfloor \sqrt{n} \rfloor - c
\end{equation*}
for every $n \in \IN$.
We obtain $\dim(x) = 1$.
\end{proof}

We conclude this section with another interesting example.

\begin{example}
	Let $x\in [0,1]$ be a left-computable number which is nearly computable but not computable (see Remark \ref{remark:lnc}). Then $x$ is strongly Kurtz random, hence not reordered computable\footnote{This fact was observed first by Janicki~\cite[Theorem 23]{Jan2024}, who gave a direct proof.}, according to Theorem~\ref{theorem:main-eq}. Using Corollary~\ref{kor:Dim<1=>roc}, we conclude that $\mathrm{Dim}(x)=1$. Furthermore, $x$ is weakly $1$-generic (see Remark \ref{remark:weakly-1-generic}). As Barmpalias and Vlek \cite[Theorem 11]{BV2011} showed that weakly $1$-generic reals are infinitely often $K$-trivial, we conclude that $\dim(x)=0$.
\end{example}

\section{Regular Numbers, Immunity, and Kurtz Randomness}
\label{section:regular}

In this section we have a short look at regular reals in the context of immunity properties, Kolmogorov complexity and (strong) Kurtz randomness.

\begin{definition}[{Wu \cite{Wu2005}}]
A real is called \emph{regular} if it can be written as a sum of finitely many strongly left-computable reals.
\end{definition}

Obviously, every regular real is left-computable.
Every regular real is reordered computable but the converse is not true~\cite{Jan2024}. It is stated in the proof of \cite[Proposition 17]{HJMS2024} that the Kolmogorov complexity of the initial segments of a regular real is at most logarithmic. For completeness sake, we add a proof of this statement.

\begin{proposition}[{\cite{HJMS2024}}]
\label{proposition:regular-Kolmogorovcomplexity}
Let $x$ be a regular real in $[0,1]$. 
Then there exists a constant $c \in \IN$ with $K(x \restriction n) \leq 2 \cdot \log(n) + 2 \cdot \log(\log(n)) + c$ for every $n \in \IN$.
\end{proposition}	

\begin{proof}
It is well-known that for every computably enumerable set $A \subseteq \IN$ there is some constant $c \in \IN$ with $K(A \restriction n) \leq 2 \cdot \log(n) + 2 \cdot \log(\log(n)) + c$ for every $n \in \IN$. 
The following is also known:
if $z\in\SigmaN$ is a real that is the sum of two left-computable reals $x,y\in\SigmaN$
then there exists a constant $c \in \IN$ with
\begin{equation*}
      K(z\restriction n) \leq \max\{ K(x \restriction n), K(y \restriction n) \}  + c
\end{equation*}
for every $n \in \IN$~\cite[Theorem 6.4]{DHNS2003}.
By applying this finitely many times we obtain the desired assertion.
\end{proof}

\begin{corollary}
	For every regular real $x\in [0,1]$ we have $\dim(x) = \mathrm{Dim}(x) = 0$.
\end{corollary}

By Theorem~\ref{theorem:Dim<1-notstronglyKurtzrandom} a regular real cannot be strongly Kurtz random. 
The question arises whether a regular real can at least be Kurtz random. We will see below that this is not the case either.

In~\cite{HJ2025} we, the authors, already looked at regular reals and immunity properties.
We proved the following two facts:
\begin{itemize}
\item
There exists a regular real in $[0,1]$ that is immune \cite[Corollary 2]{HJ2025}.
\item
However, a regular real in $[0,1]$ cannot be hyperimmune \cite[Theorem 3]{HJ2025}.
\end{itemize}
Can a regular real be bi-immune? The answer is no.

\begin{theorem}
\label{theorem:regular-not-biimmune}
Regular reals in $[0,1]$ are not bi-immune.
\end{theorem}

\begin{proof}
Wu~\cite[Theorem 2.1]{Wu2005} has shown that any regular real in $[0,1]$ is of the form $x_A$ for a set $A\subseteq\IN$ that is 
an element of the field of sets generated by the c.e.~sets.
It was shown by Markwald~\cite[Satz 6]{Mar1956}
that such sets cannot be bi-immune.
\end{proof}

\begin{corollary}
\label{corollary:regular-notKurtzrandom}
Regular reals in $[0,1]$ are not Kurtz random.
\end{corollary}

\begin{proof}
This follows from Theorem~\ref{theorem:regular-not-biimmune}
and from the fact that Kurtz random reals are bi-immune 
(this is due to Jockusch; see \cite[Theorem 7.2.23]{DH2010}).
\end{proof}

Barmpalias and Vlek~\cite[Prop.~2.2]{BV2011} have shown that every strongly left-computable number in $[0,1]$ is infinitely often $K$-trivial. The following result strengthens this observation and complements Proposition~\ref{proposition:regular-Kolmogorovcomplexity}.

\begin{theorem}
Every regular real in $[0,1]$ is infinitely often $K$-trivial. 
\end{theorem}

\begin{proof}
Let us consider an $x\in\Sigma^\IN$ such that $0.x$ is a regular real.
We can assume without loss of generality $0.x>0$.
Then there exist a constant $k\in\IN$ and a computable name $f$ of $0.x$ such that $u_f(n)\leq k$ for all $n\in\IN$;
compare \cite[Lemma 9.4]{HHJ2024}.
We define $x_t:=\sum_{i=0}^{t-1} 2^{-f(i)}$. 
The increasing sequence of \emph{true stages} $(s_m)_m$ for $f$ is defined as usual:
\begin{itemize}
\item
Let $m_0:=\min f(\IN)$ and
$s_0:=\max\{t\in\IN \mid f(t)=m_0\}$.
\item
Once $s_0,\ldots,s_n$ are defined, let $m_{n+1}:=\min f(\{s(n)+1,s(n)+2,\ldots\})$
and $s_{n+1} := \max\{t\in\IN \mid f(t)=m_{n+1}\}$.
\end{itemize}
Then $f(t) > f(s_n)$ for all $n\in\IN$ and all $t > s_n$.
Let $\ell := \limsup_n u_f(f(s_n))$. Of course, $1\leq\ell\leq k$.
The set $B:=\{f(s_n) \mid u_f(f(s_n)) = \ell\}$ is an infinite set.
We claim that there exists a constant $d\in\IN$ such that
$K(x \restriction m) \leq K(0^m) + d$ for all $m\in B$.

Remember that by $U$ we denote some fixed optimal prefix-free machine, and that we write $K(\tau)$ for $K_U(\tau)$. 
We wish to define a machine $M$ with prefix-free domain. For $\sigma \in \SigmaS$,
we define $M(\sigma)$ as follows:
\begin{itemize}
\item
If $U(\sigma)$ is not defined or $U(\sigma)\not\in\{0^m \mid m\in\IN\}$ 
then we leave $M(\sigma)$ undefined.
\item
If $U(\sigma)=0^m$ for some $m\in\IN$ then we search for a number $t$ with $f(t)=m$
and such that there are exactly $\ell$ numbers $r$ with $0\leq r\leq t$ and $f(r)=m$.
If there is no such $t$ then we leave $M(\sigma)$ undefined.
\item
But if there is such a $t$ (there can be at most one) then we let $M(\sigma)$
be the string consisting of the first $m$ bits after the binary point of the binary representation of $x_{t+1}$.
\end{itemize}
It is clear that the domain of definition of $M$ is a subset of the domain of definition of $U$.
Hence, $M$ has prefix-free domain.
As $U$ is optimal, there exists a constant $d_1\in\IN$ such that
$K_U(\tau)\leq K_M(\tau) + d_1$ for all $\tau\in\Sigma^*$.
Let us consider a number $m\in B$.
Let us fix a shortest string $\sigma$ with $U(\sigma)=0^m$.
Let $n$ be the unique number with $f(s_n)=m$. 
As $M(\sigma)$ is obtained from $x_{s_n+1}$ by keeping only the first
$m$ digits after the binary point we see 
$\left| 0.M(\sigma) - x_{s_n+1} \right| \leq 2^{-m}$.
Similarly, $|0.(x\restriction m) - x| \leq 2^{-m}$.
As $f(t)>f(s_n)=m$ for all $t>s_n$ we obtain
\[ \left| x - x_{s_n+1} \right| \leq k\cdot 2^{-m}. \]
Together, these estimates give us:
\[ \left| 0.M(\sigma) - 0.(x\restriction m) \right| \leq (k+2) \cdot 2^{-m} . \]
By~\cite[Lemma 4.8]{CHKW2001} there exists a constant $d_2\in\IN$
such that, for all $\ell\in\IN$ and all $\rho,\tau\in\Sigma^\ell$
with $\left| 0.\rho - 0.\tau \right| \leq (k+2)\cdot 2^{-\ell}$ we have
$\left| K(\sigma) - K(\tau) \right| \leq d_2$.
This implies
\[ \left| K(M(\sigma)) - K( x\restriction m ) \right| \leq d_2 . \]
By combining these estimates, we conclude
\begin{eqnarray*}
  K(x \restriction m) 
  &\leq & K(M(\sigma)) + d_2
	\leq K_M(M(\sigma)) + d_1 + d_2 \\ 
	&\leq & |\sigma| + d_1 + d_2 
	= K(0^m) + d_1 + d_2. 
\end{eqnarray*}
This ends the proof.
\end{proof}

\bibliography{rrc}
\bibliographystyle{abbrv}

\end{document}